\documentclass{amsart}

\usepackage{graphicx,epsfig}

\usepackage{amsmath,amssymb,latexsym, amsfonts, amscd, amsthm, xy}

\usepackage{url}

\usepackage{hyperref}

\input{xy}
\xyoption{all}

\usepackage[usenames]{color}

\makeindex 

=601pt \headheight = 12pt \headsep = 20pt

\hypersetup{
    colorlinks   = true,
      citecolor    = red
}

\hypersetup{linkcolor=red}

\hypersetup{linkcolor=blue}

\theoremstyle{plain}
\newtheorem{theorem}{Theorem}[section]

\newtheorem{proposition}[theorem]{Proposition}
\newtheorem{corollary}[theorem]{Corollary}

\newtheorem{lemma}[theorem]{Lemma}

\theoremstyle{definition}

\newtheorem{remark}[theorem]{Remark}

\def\bF{\mathbb{F}}

\def\bK{\mathbb{K}}

\def\bQ{\mathbb{Q}}
\def\bZ{\mathbb{Z}}

\def\A{\mathbf{A}}

\def\cG{\mathcal{G}}

\def\cO{\mathcal{O}}

\def\cP{\mathcal{P}}

\def\cQ{\mathcal{Q}}

\def\cR{\mathcal{R}}

\def\cZ{\mathcal{Z}}

\def\cU{\mathcal{U}}

\def\cT{\mathcal{T}}

\def\fp{\mathfrak{p}}

\def\fq{\mathfrak{q}}

\def\deg{\mathbf{deg}}

\def\Gal{\mathrm{Gal}}

\def\Norm{\mathrm{Norm}}

\def\fc{\mathfrak{c}}

\def\fg{\mathfrak{g}}

\def\k{\mathbf{k}}

\def\alg{\mathbf{alg}}

\def\Nm{\mathbf{Norm}}

\begin{document}

\title{Representation of units in cyclotomic function fields}

\author{Nguyen Ngoc Dong Quan}

\date{December 16, 2015}

\address{Department of Mathematics \\
         The University of Texas at Austin \\
         Austin, TX 78712 \\
         USA}

\email{\href{mailto:dongquan.ngoc.nguyen@gmail.com}{\tt dongquan.ngoc.nguyen@gmail.com}}

\maketitle

\tableofcontents

\section{Introduction}

Let $q$ be be a power of a prime $p$, and let $\bF_q$ denote the finite field with $q$ elements. Let $\A = \bF_q[T]$ be the ring of polynomials in the variable $T$ over $\bF_q$, and let $\k = \bF_q(T)$ be the quotient field of $\A$. Let $\k^{\alg}$ denote an algebraic closure of $\k$. Let $\tau$ be the mapping defined by $\tau(x) = x^q$, and let $\k\langle \tau \rangle$ denote the twisted polynomial ring. Let $C : \A \rightarrow \k\langle \tau \rangle$ ($a \mapsto C_a$) be the Carlitz module, namely, $C$ is an $\bF_q$-algebra homomorphism such that $C_T = T + \tau$.

Let $m$ be a polynomial of positive degree, and set $\Lambda_m = \{\lambda \in \k^{\alg} \; | \; C_m(\lambda) = 0 \}$. We define a \textit{primitive $m$-th root of $C$} to be a root of the polynomial $C_m(x) \in \A[x]$ that generates the $\A$-module $\Lambda_m$. Throughout the paper, for each polynomial $m$, we fix a primitive $m$-th root of $C$, and denote it by $\lambda_m$. The \textit{$m$-th cyclotomic function field}, denoted by $\bK_m$, is defined by $\bK_m = \k(\Lambda_m) = \k(\lambda_m)$. It is known (see Hayes \cite{Hayes}, or Rosen \cite{Rosen}) that $\bK_m$ is a field extension of $\k$ of degree $\Phi(m)$, where $\Phi(\cdot)$ is a function field analogue of the classical Euler $\phi$-function. (We will recall the definition of $\Phi(\cdot)$ in Section \ref{S-basic-notions}.)

There are many strong analogies between the $m$-th cyclotomic function fields $\bK_{m}$ and the classical cyclotomic fields $\bQ(\zeta_m)$ (see Goss \cite{Goss}, Hayes \cite{Hayes}, Rosen \cite{Rosen}, or Thakur \cite{Thakur}). Here the former letter $m$ denotes a polynomial of positive degree in $\A$, and the latter letter $m$ stands for a positive integer in $\bZ$. In this paper, we are interested in studying new analogous phenomena between the cyclotomic function fields and the classical cyclotomic fields. 

Newman \cite{Newman} proved a refinement of Hilbert's Satz 90 for the extensions $\bQ(\zeta_p)/\bQ$, where $p > 3$ is a prime, and $\zeta_p$ denotes the $p$-th root of unity. Recall that for a unit $\epsilon$ of norm $1$ in $\bZ[\zeta_p]$, Hilbert's Satz 90 (see Lang \cite{Lang}) tells us that one can write $\epsilon$ as a quotient of conjugate integers in $\bZ[\zeta_p]$, i.e., there exists an element $\delta \in \bZ[\zeta_p]$ and an element $\sigma \in \Gal(\bQ(\zeta_p)/\bQ)$ such that $\epsilon = \dfrac{\delta}{\sigma(\delta)}$. Newman (see \cite[Corollary, p.357]{Newman}) gave a refinement of Hilbert's Satz 90 by proving a sufficient and necessary condition for which such an element $\delta$ can be chosen to be a unit in $\bZ[\zeta_p]$, and thus $\epsilon$ is a quotient of conjugate units in $\bZ[\zeta_p]$. In order to obtain this result, Newman proved a stronger result that provides a unique representation of a unit of norm $1$ as a product of a power of a unit $\dfrac{1 - \zeta_p^e}{1 - \zeta_p}$ with a quotient of conjugate units. More precisely, Newman proved the following.

\begin{theorem}
\label{Newman-T}
$(\text{Newman, see \cite[Theorem, p.353]{Newman}})$

Let $e$ be a primitive root modulo $p$, and let $\sigma_e$ be a generator of the Galois group $\Gal(\bQ(\zeta_p)/\bQ)$ that sends $\zeta_p$ to $\zeta_p^e$. Let $\cG_p$ be the multiplicative group consisting of all units of the form $\dfrac{\delta}{\sigma_e(\delta)}$, where $\delta$ is a unit in $\bZ[\zeta_p]$. Then for any unit $\epsilon$ of norm $1$ in $\bZ[\zeta_p]$, there exist an integer $\ell$ with $0 \le \ell \le p - 2$, and a unit $\alpha \in \cG_p$ such that
\begin{align*}
\epsilon = \left(\dfrac{1 - \zeta_p^e}{1 - \zeta_p}\right)^{\ell}\alpha.
\end{align*}
Furthermore the above representation is unique, i.e., if $\epsilon = \left(\dfrac{1 - \zeta_p^e}{1 - \zeta_p}\right)^{\ell_{\star}}\alpha_{\star}$ for some integer $\ell_{\star}$ with $0 \le \ell_{\star} \le p - 2$ and some unit $\alpha_{\star} \in \cG_p$, then $\ell_{\star} = \ell$ and $\alpha_{\star} = \alpha$.

\end{theorem}

The aim of this paper is to prove a function field analogue of the above theorem, which can be viewed as a refinement of Hilbert's Satz 90 for the extensions $\bK_{\wp}/\k$. More explicitly, our main goal in this paper is to prove the following.

\begin{theorem}
\label{MT-in-S1}
$(\text{See Theorem \ref{MT} and Corollary \ref{C1-MT}})$

Let $\wp$ be a monic prime in $\A$. Let $\bK_{\wp}$ be the $\wp$-th cyclotomic function field, and let $\cO_{\wp}$ be the ring of integers of $\bK_{\wp}$. Let $\fg$ be a primitive root modulo $\wp$, and let $\sigma_{\fg}$ be a generator of the Galois group $\Gal(\bK_{\wp}/\k)$ that sends $\lambda_{\wp}$ to $C_{\fg}(\lambda_{\wp}$. Let $\cG_{\wp}$ be the multiplicative group consisting of all units of the form $\dfrac{\delta}{\sigma_{\fg}(\delta)}$, where $\delta$ is a unit in $\cO_{\wp}$. Then for any unit $\epsilon$ of norm $1$ in $\cO_{\wp}$, there exist an integer $\ell$ with $0 \le \ell \le q^{\deg(\wp)} - 2$, and a unit $\alpha \in \cG_{\wp}$ such that
\begin{align*}
\epsilon = \left(\dfrac{\lambda_{\wp}}{C_{\fg}(\lambda_{\wp})} \right)^{\ell}\alpha.
\end{align*}
Furthermore the above representation is unique, i.e.m if $\epsilon = \left(\dfrac{\lambda_{\wp}}{C_{\fg}(\lambda_{\wp})} \right)^{\ell_{\star}}\alpha_{\star}$ for some integer $\ell_{\star}$ with $0 \le \ell_{\star} \le q^{\deg(\wp)} - 2$ and some unit $\alpha_{\star} \in \cG_{\wp}$, then $\ell_{\star} = \ell$ and $\alpha_{\star} = \alpha$.

\end{theorem}

One can replace $\dfrac{1 - \zeta_p^e}{1 - \zeta_p}$ in Theorem \ref{Newman-T} by its inverse $\dfrac{1 - \zeta_p}{1 - \zeta_p^e}$, and obtain a similar representation for units of norm $1$ that is equivalent to that of Theorem \ref{Newman-T}. It is well-known (see Rosen \cite[Proposition 12.6]{Rosen}) that $\dfrac{\lambda_{\wp}}{C_{\fg}(\lambda_{\wp})}$ is a unit in $\cO_{\wp}$, and is a function field analogue of the unit $\dfrac{1 - \zeta_p}{1 - \zeta_p^e}$ in the number field setting. Hence Theorem \ref{MT-in-S1} can be viewed as a function field analogue of Newman's theorem. 

Using the same ideas as in Newman \cite{Newman}, we also obtain a sufficient and necessary condition for which a unit of norm $1$ in $\cO_{\wp}$ can be written as a quotient of conjugate units. This is a refinement of Hilbert's Satz 90 for the extensions $\bK_{\wp}/\k$. 

The structure of the paper is as follows. In Section \ref{S-basic-notions}, we recall some basic notions and notation that will be used throughout the paper. In Section \ref{S-main-results}, we prove Theorem \ref{MT-in-S1} (see Theorem \ref{MT} and Corollary \ref{C1-MT}), and consequently obtain a refinement of Hilbert's Satz 90 (see Corollary \ref{C3-MT}). The proof of Theorem \ref{MT-in-S1} uses the same ideas and approach as in the proof of Newman \cite[Theorem]{Newman}, but we need to introduce some modifications in many places to adapt Newman's proof of \cite[Theorem]{Newman} into the function field setting.

\section{Some basic notions} 
\label{S-basic-notions}

The aim of this section is to recall some basic notions and fix some notation that that will be used throughout the paper.

For a polynomial $m \in \A$ of positive degree, we define $\Phi(m)$ to be the number of nonzero polynomials of degree less than $\deg(m)$ and relatively prime to $m$. The function $\Phi(\cdot)$ is a function field analogue of the classical Euler $\phi$-function.

Let $m = \alpha \wp_1^{s_1}\cdots \wp_h^{s_h}$ be the prime factorization of $m$, where $\alpha \in \bF_q^{\times}$, the $\wp_i$ are monic primes in $\A$, and the $s_i$ are positive integers. It is well-known (see \cite[Proposition 1.7]{Rosen}) that
\begin{align*}
\Phi(m) = \prod_{i = 1}^h\Phi(\wp_i^{s_i}) = \prod_{i = 1}^h(q^{\deg(\wp_i^{s_i})} - q^{\deg(\wp_i^{s_i - 1})}).
\end{align*}
In particular, when $m = \wp^s$ for some monic prime $\wp$ and some positive integer $s$, 
\begin{align*}
\Phi(\wp^s) = q^{\deg(\wp^{s})} - q^{\deg(\wp^{s - 1})}.
\end{align*}

Recall that the \textit{$m$-th cyclotomic polynomial}, denoted by $\Psi_m(x)$, is the minimal polynomial of $\lambda_m$ over $\k$. It is well-known (see Hayes \cite{Hayes}, or Rosen \cite{Rosen}) that $\Psi_m(x)$ is the monic irreducible polynomial of degree $\Phi(m)$ with coefficients in $\A$ such that $\Psi_m(\lambda_m) = 0$.

When $m = \wp^s$ for some monic prime $\wp$ and some positive integer $s$, we know from \cite[Proposition 2.4]{Hayes} that 
\begin{align}
\label{Eqn-the-cyclotomic-polynomial-at-a-power-of-prime}
\Psi_{\wp^s}(x) = C_{\wp^s}(x)/C_{\wp^{s - 1}}(x).
\end{align}  

From Rosen \cite[Proposition 12.11]{Rosen}, one can write
\begin{align*}
C_{\wp}(x) = \wp x + [\wp, 1]x^q + \ldots, + [\wp, \deg(\wp) - 1]x^{q^{\deg(\wp) - 1}} + x^{q^{\deg(\wp)}},
\end{align*}
where $[\wp, i] \in \A$ for each $1 \le i \le \deg(\wp) - 1$. For $s = 1$, the equation (\ref{Eqn-the-cyclotomic-polynomial-at-a-power-of-prime}) tells us that
\begin{align}
\label{Eqn-the-cyclotomic-polynomial-at-prime}
\Psi_{\wp}(x) = \dfrac{C_{\wp}(x)}{x} = \wp  + [\wp, 1]x^{q - 1} + \ldots, + [\wp, \deg(\wp) - 1]x^{q^{\deg(\wp) - 1} - 1} + x^{q^{\deg(\wp)} - 1}.
\end{align}

When $x = 0$, we obtain the following elementary result that will be useful in the proof of our main theorem.

\begin{proposition}
\label{P-Psi(0)=wp}

$\Psi_{\wp}(0) = \wp$.

\end{proposition}

\subsection{The Galois group $\Gal(\bK_{\wp}/\k)$}
\label{Subsection-Galois-group}

For the rest of this paper, fix a monic prime $\wp$ of positive degree. Let $\bK_{\wp}$ be the $\wp$-th cyclotomic function field, and let $\cO_{\wp}$ be the ring of integers of $\bK_{\wp}$. To rule out the trivialities, we further assume that 
\begin{itemize}

\item [$(\star)$] $q > 2$ or $\deg(\wp) > 1$. 

\end{itemize}

We denote by $\Gal(\bK_{\wp}/\k)$ the Galois group of $\bK_{\wp}/\k$. There is an isomorphism between $\Gal(\bK_{\wp}/\k)$ and the multiplicative group $(\A/\wp\A)^{\times}$ (see Rosen \cite[Chapter 12]{Rosen}). For each element $m \in (\A/\wp\A)^{\times}$, there exists an isomorphism $\sigma_m \in \Gal(\bK_{\wp}/\k)$ that is uniquely determined by the relation $\sigma_m(\lambda_{\wp}) = C_m(\lambda_{\wp})$. The correspondence $\sigma_m \mapsto m$ is an isomorphism from $\Gal(\bK_{\wp}/\k)$ into the group $(\A/\wp\A)^{\times}$.

Throughout the paper, fix an element $\fg \in \A$ such that $\fg$ is a primitive root modulo $\wp$, i.e., $\fg$ is a generator of the group $(\A/\wp\A)^{\times}$. Note that the order of $(\A/\wp\A)^{\times}$ is $\Phi(\wp) = q^{\deg(\wp)} - 1$, and thus one can write $(\A/\wp\A)^{\times} = \langle \fg \rangle = \{\fg^e \; | \; 0 \le e \le q^{\deg(\wp)} - 2 \}$. We will prove that the isomorphism $\sigma_{\fg} \in \Gal(\bK_{\wp}/\k)$ is a generator of the cyclic group $ \Gal(\bK_{\wp}/\k)$, i.e., for each $0 \le e \le q^{\deg(\wp)} - 2$, $\sigma_{\fg^e}(\lambda_{\wp}) = \sigma_{\fg}^e(\lambda_{\wp})$. Indeed, take an arbitrary integer $e$ with $2 \le e \le q^{\deg(\wp)} - 2$. By induction, we see that
\begin{align*}
\sigma_{\fg^e}(\lambda_{\wp}) = C_{\fg^e}(\lambda_{\wp}) &= C_{\fg^{e - 1}}(C_{\fg}(\lambda_{\wp})) \\
&= C_{\fg^{e - 1}}(\sigma_{\fg}(\lambda_{\wp})) \\
&= \sigma_{\fg}(C_{\fg^{e - 1}}(\lambda_{\wp})) \\
&= \sigma_{\fg}(\sigma_{\fg^{e - 1}}(\lambda_{\wp})) \\
&= \sigma_{\fg}(\sigma_{\fg}^{e - 1}(\lambda_{\wp})) \\
&= \sigma_{\fg}^e(\lambda_{\wp}).
\end{align*}

We summarize the above discussion in the following.

\begin{proposition}
\label{P1}

The isomorphism $\sigma_{\fg} \in \Gal(\bK_{\wp}/\k)$ is a generator for the cyclic group $\Gal(\bK_{\wp}/\k)$. More precisely, $\sigma_{\fg^e}(\lambda_{\wp}) = \sigma_{\fg}^e(\lambda_{\wp})$ for each $0 \le e \le q^{\deg(\wp)} - 2$.

\end{proposition}

Throughout the paper, we denote by $\Nm_{\bK_{\wp}/\k}(\cdot)$ the norm from $\bK_{\wp}$ to $\k$. Proposition \ref{P1} implies that for each element $\alpha \in \bK_{\wp}^{\times}$, 
\begin{align*}
\Nm_{\bK_{\wp}/\k}(\alpha) = \prod_{\sigma \in \Gal(\bK_{\wp}/\k)}\sigma(\alpha) =  \prod_{e = 0}^{q^{\deg(\wp)} - 2}\sigma_{\fg}^e(\alpha).
\end{align*}

 The following elementary result will be useful in the proof of our main theorem.

\begin{proposition}
\label{P-Hilbert-Satz-90-for-units}

Let $\alpha$ be a unit in $\cO_{\wp}$ of norm $1$. Then there exists an element $\delta \in \cO_{\wp}$ such that
\begin{align*}
\alpha = \dfrac{\delta}{\sigma_{\fg}(\delta)}.
\end{align*}

\end{proposition}

\begin{proof}

Applying Hilbert's Satz 90 (see Lang \cite{Lang}), there exists an element $\gamma \in \bK_{\wp}$ such that
\begin{align}
\label{E1-in-Hilbert-Satz-90}
\alpha = \dfrac{\gamma}{\sigma_{\fg}(\gamma)}.
\end{align}
Since $\bK_{\wp} = \k(\lambda_{\wp})$, one can write 
\begin{align*}
\gamma = \dfrac{\sum_{i}a_i\lambda_{\wp}^i}{b},
\end{align*}
where the $a_i$ are in $\A$, and $b \in \A^{\times}$. From \eqref{E1-in-Hilbert-Satz-90}, one gets
\begin{align*}
\alpha = \dfrac{ \dfrac{\sum_{i}a_i\lambda_{\wp}^i}{b}}{\sigma_{\fg}\left( \dfrac{\sum_{i}a_i\lambda_{\wp}^i}{b}\right)} = \dfrac{\delta}{\sigma_{\fg}(\delta)},
\end{align*}
where $\delta = \sum_{i}a_i\lambda_{\wp}^i \in \A[\lambda_{\wp}]$. Since $\cO_{\wp} = \A[\lambda_{\wp}]$ (see Rosen \cite[Proposition 12.9]{Rosen}), our contention follows.

\end{proof}

\subsection{Representation of integers in $\cO_{\wp}$, and the multiplicative group $\cG_{\wp}$}
\label{Subsection-G-wp}

It is well-known (see \cite[Proposition 12.9]{Rosen}) that $\cO_{\wp} = \A[\lambda_{\wp}]$. For each polynomial $\cP(x) \in \A[x]$ of degree $\le q^{\deg(\wp)} - 2$, $\cP(\lambda_{\wp})$ is an integer in $\cO_{\wp}$. Conversely we will prove that for each integer $\alpha \in \cO_{\wp}$, there exists a unique polynomial $\cP_{\alpha}(x) \in \A[x]$ of degree $\le q^{\deg(\wp)} - 2$ such that $\alpha = \cP_{\alpha}(\lambda_{\wp})$. Indeed we know that the minimal polynomial of $\lambda_{\wp}$ is the $\wp$-th cyclotomic polynomial $\Psi_{\wp}(x) \in \A[x]$. From (\ref{Eqn-the-cyclotomic-polynomial-at-a-power-of-prime}), $\Psi_{\wp}(x)$ can be explicitly written in the form
\begin{align*}
\Psi_{\wp}(x) = \dfrac{C_{\wp}(x)}{x},
\end{align*}
and thus $\deg(\Psi_{\wp}(x)) = q^{\deg{\wp}} - 1$. Hence $\{\lambda_{\wp}^e\}_{0 \le e \le q^{\deg(\wp)} - 2}$ is a basis for the $\A$-module $\A[\lambda_{\wp}]$. Hence for each integer $\alpha \in \cO_{\wp}$, there exists a polynomial $\cP_{\alpha}(x) \in \A[x]$ of degree  $\le q^{\deg(\wp)} - 2$ such that $\alpha = \cP_{\alpha}(\lambda_{\wp})$. We prove that $\cP_{\alpha}$ is unique. Assume the contrary, i.e., there exists a polynomial $\cQ_{\alpha}(x) \in \A[x]$ of degree $\le q^{\deg(\wp)} - 2$ such that $\cQ_{\alpha}(x) \ne \cP_{\alpha}(x)$, and $\alpha = \cQ_{\alpha}(\lambda_{\wp})$. Setting $F(x) =  \cP_{\alpha}(x) - \cQ_{\alpha}(x) \in \A[x]$, we deduce that
\begin{align*}
F(\lambda_{\wp}) = \cP_{\alpha}(\lambda_{\wp}) - \cQ_{\alpha}(\lambda_{\wp}) = 0.
\end{align*}
Since $\cQ_{\alpha}(x) \ne \cP_{\alpha}(x)$, the polynomial $F(x)$ is nonzero. Furthermore $F(x)$ is of degree at most $q^{\deg(\wp)} - 2$, which is a contradiction since $\Psi_{\wp}(x)$ is the minimal polynomial of $\lambda_{\wp}$ and $\deg(\Psi_{\wp}(x)) = q^{\deg{\wp}} - 1 > q^{\deg(\wp)} - 2$. 

We summarize the above discussion in the following.

\begin{proposition}
\label{P2}

For each integer $\alpha \in \cO_{\wp}$, there exists a unique polynomial $\cP_{\alpha}(x) \in \A[x]$ of degree at most $q^{\deg(\wp)} - 2$ such that $\alpha = \cP_{\alpha}(\lambda_{\wp})$. Furthermore $\alpha = 0$ if and only if the polynomial $\cP_{\alpha}(x)$ is identical to zero.

\end{proposition}

For the rest of the paper, for each integer $\alpha \in \cO_{\wp}$, we always denote by $\cP_{\alpha}$ the unique polynomial satisfying $\alpha = \cP_{\alpha}(\lambda_{\wp})$ in Proposition \ref{P2}. We call $\cP_{\alpha}$ the \textit{polynomial representing $\alpha$}.

\begin{proposition}
\label{P3}

Let $\alpha$ be a nonzero element in $\cO_{\wp}$, and let $\cP_{\alpha}(x) \in \A[x]$ be the polynomial representing $\alpha$. If $\cQ(x)$ is a polynomial in $\A[x]$ such that $\cQ(\lambda_{\wp}) = \alpha$, then
\begin{align*}
\dfrac{\cQ(\lambda_{\wp})}{\cQ(C_{\fg}(\lambda_{\wp}))} = \dfrac{\cP_{\alpha}(\lambda_{\wp})}{\cP_{\alpha}(C_{\fg}(\lambda_{\wp}))} = \dfrac{\alpha}{\sigma_{\fg}(\alpha)}.
\end{align*}

\end{proposition}

\begin{remark}
\label{R1}

Since $\alpha = \cQ(\lambda_{\wp}) \ne 0$, and 
\begin{align*}
\cQ(C_{\fg}(\lambda_{\wp})) = \cQ(\sigma_{\fg}(\lambda_{\wp})) = \sigma_{\fg}(\cQ(\lambda_{\wp})),
\end{align*}
$\cQ(C_{\fg}(\lambda_{\wp}))$ is also nonzero. Similarly $\cP(C_{\fg}(\lambda_{\wp}))$ is nonzero.

\end{remark}

\begin{proof}

From the assumption and Remark \ref{R1}, we know that 
\begin{align*}
\dfrac{\cQ(\lambda_{\wp})}{\cQ(C_{\fg}(\lambda_{\wp}))} = \dfrac{\alpha}{\sigma_{\fg}(\alpha)}.
\end{align*}
Similarly one also gets 
\begin{align*}
 \dfrac{\cP_{\alpha}(\lambda_{\wp})}{\cP_{\alpha}(C_{\fg}(\lambda_{\wp}))} = \dfrac{\alpha}{\sigma_{\fg}(\alpha)},
 \end{align*}
 and thus Proposition \ref{P3} follows.

\end{proof}

We now introduce the multiplicative group $\cG_{\wp}$ that will be of interest in this paper. Let $\cG_{\wp}$ be the set consisting of all elements of the form $\dfrac{\cQ(\lambda_{\wp})}{\cQ(C_{\fg}(\lambda_{\wp}))}$, where $\cQ(x) \in \A[x]$ such that $\cQ(\lambda_{\wp})$ is a unit in $\cO_{\wp}$. From Proposition \ref{P3}, we see that all the elements of $\cG_{\wp}$ are units in $\cO_{\wp}$. Furthermore the set $\cG_{\wp}$ is invariant under the multiplication, i.e., $\epsilon \gamma \in \cG_{\wp}$ for any $\epsilon, \gamma \in \cG_{\wp}$. For an arbitrary element $\dfrac{\cQ(\lambda_{\wp})}{\cQ(C_{\fg}(\lambda_{\wp}))} \in \cG_{\wp}$, where $\cQ(x) \in \A[x]$ such that $\alpha = \cQ(\lambda_{\wp})$ is a unit in $\cO_{\wp}$. Note that $\alpha^{-1}$ is a unit in $\cO_{\wp}$. Let $\cP_{\alpha^{-1}}(x) \in \A[x]$ be the polynomial representing $\alpha^{-1}$. From Proposition \ref{P3}, $\dfrac{\cP_{\alpha^{-1}}(\lambda_{\wp})}{\cP_{\alpha^{-1}}(C_{\fg}(\lambda_{\wp}))} = \dfrac{\alpha^{-1}}{\sigma_{\fg}(\alpha^{-1})}$, and thus $\dfrac{\alpha^{-1}}{\sigma_{\fg}(\alpha^{-1})}$ belongs to $\cG_{\wp}$. By Proposition \ref{P3}, we deduce that
\begin{align*}
\dfrac{\cQ(\lambda_{\wp})}{\cQ(C_{\fg}(\lambda_{\wp}))}\dfrac{\alpha^{-1}}{\sigma_{\fg}(\alpha^{-1})} = \dfrac{\alpha}{\sigma_{\fg}(\alpha)}\dfrac{\alpha^{-1}}{\sigma_{\fg}(\alpha^{-1})} = 1,
\end{align*}
which proves that $\dfrac{\alpha^{-1}}{\sigma_{\fg}(\alpha^{-1})}$ is the inverse of $\dfrac{\cQ(\lambda_{\wp})}{\cQ(C_{\fg}(\lambda_{\wp}))}$. Hence $\cG_{\wp}$ is a multiplicative group.

Throughout this paper, we denote by $\cU_{\star}(\cO_{\wp})$ be the group of units of norm 1 in $\cO_{\wp}$. The following result follows immediately from the above discussion and Proposition \ref{P3}.

\begin{proposition}
\label{P4}

\begin{itemize}

\item []

\item [(i)] $\cG_{\wp}$ is a subgroup of $\cU_{\star}(\cO_{\wp})$.

\item [(ii)] 
\begin{align*}
\cG_{\wp} = \left\{\dfrac{\alpha}{\sigma_{\fg}(\alpha)} \; | \; \text{$\alpha$ is a unit in $\cO_{\wp}$} \right\}.
\end{align*}

\item [(iii)]
\begin{align*}
\cG_{\wp} = \left\{\dfrac{\cP_{\alpha}(\lambda_{\wp})}{\cP_{\alpha}(C_{\fg}(\lambda_{\wp}))} \; | \; \text{$\alpha$ is a unit in $\cO_{\wp}$, and $\cP_{\alpha}(x) \in \A[x]$ is the polynomial representing $\alpha$} \right\}.
\end{align*}

\end{itemize}

\end{proposition}

\section{Representation of units}
\label{S-main-results}

In this section, we will prove Theorem \ref{MT-in-S1} (see Theorem \ref{MT} and Corollary \ref{C1-MT}). As a consequence, we obtain a refinement of Hilbert's Satz 90 (see Corollary \ref{C3-MT}). We begin by proving several lemmas that we will need in the proof of our main theorem.

The next result is a function field analogue of Newman \cite[Lemma 1]{Newman}. We follow the same ideas as in the proof of Newman \cite[Lemma 1]{Newman} to prove the next lemma. 

\begin{lemma}
\label{L1}

Let $\alpha$ be an integer in $\cO_{\wp}$, and let $\cP_{\alpha}(x) \in \A[x]$ be the polynomial representing $\alpha$. Assume that the following are true.
\begin{itemize}

\item [(i)] $\dfrac{\cP_{\alpha}(C_{\fg}(\lambda_{\wp}))}{\cP_{\alpha}(\lambda_{\wp})}$ is a unit in $\cO_{\wp}$. $(\text{Recall that $\fg$ is a generator of the group $(\A/\wp\A)^{\times}$.})$

\item [(ii)] $\gcd(\cP_{\alpha}(0), \wp) = 1$.

\item [(iii)] the content of the polynomial $\cP_{\alpha}(x)$ is $1$, i.e., the greatest common divisor of all coefficients of $\cP_{\alpha}(x)$ is $1$.

\end{itemize}
Then $\alpha$ is a unit in $\cO_{\wp}$. 

\end{lemma}

\begin{remark}
\label{R1-in-L1}

By (ii) in Lemma \ref{L1}, one sees that the polynomial $\cP_{\alpha}(x)$ is nonzero, and it thus follows from Proposition \ref{P2} that $\alpha = \cP_{\alpha}(\lambda_{\wp}) \ne 0$.

\end{remark}

\begin{proof}

Assume the contrary, i.e., $\alpha$ is a non-unit in $\cO_{\wp}$. Since $\alpha = \cP_{\alpha}(\lambda_{\wp})$, $\cP_{\alpha}(\lambda_{\wp})$ is also a non-unit in $\cO_{\wp}$. Hence there exists a prime ideal $\fq$ in $\cO_{\wp}$ that divides $\cP_{\alpha}(\lambda_{\wp})$. Thus the ideal $\sigma_{\fg}(\fq)$ is prime, and is a prime ideal divisor of $\sigma_{\fg}(\cP_{\alpha}(\lambda_{\wp}))$. Since  
\begin{align*}
\sigma_{\fg}(\cP_{\alpha}(\lambda_{\wp})) = \cP_{\alpha}(\sigma_{\fg}(\lambda_{\wp})) = \cP_{\alpha}(C_{\fg}(\lambda_{\wp})), 
\end{align*}
it follows that $\sigma_{\fg}(\fq)$ is a prime ideal divisor of $\cP_{\alpha}(C_{\fg}(\lambda_{\wp}))$. We deduce from (i) that $\sigma_{\fg}(\fq)$ is also a prime ideal divisor of $\cP_{\alpha}(\lambda_{\wp})$. Since $\sigma_{\fg}$ is a generator of the Galois group $\Gal(\bK_{\wp}/\k)$, every conjugate of the prime ideal $\fq$ is a prime ideal divisor of $\cP_{\alpha}(\lambda_{\wp})$, i.e., for every element $\sigma \in \Gal(\bK_{\wp}/\k)$, $\sigma(\fq)$ is a prime ideal divisor of $\cP_{\alpha}(\lambda_{\wp})$.

We contend that $\gcd(\cP_{\alpha}(\lambda_{\wp}), \lambda_{\wp}) = 1$; otherwise since $\lambda_{\wp}\cO_{\wp}$ is a prime ideal in $\cO_{\wp}$ (see Rosen \cite[Proposition 12.7]{Rosen}), we deduce that $\lambda_{\wp}\cO_{\wp}$ divides $\cP_{\alpha}(\lambda_{\wp})$, and thus
\begin{align*}
\cP_{\alpha}(0) \equiv \cP_{\alpha}(\lambda_{\wp}) \equiv 0 \pmod{\lambda_{\wp}}.
\end{align*}
Since $\wp\cO_{\wp} = (\lambda_{\wp}\cO_{\wp})^{q^{\deg(\wp)} - 1}$ (see Rosen \cite[Proposition 12.7]{Rosen}), we deduce that $\cP_{\alpha}(0) \equiv 0 \pmod{\wp}$, which is a contradiction to (ii).

Since the prime ideal $\sigma(\fq)$ divides $\cP_{\alpha}(\lambda_{\wp})$ for every $\sigma \in \Gal(\bK_{\wp}/\k)$, we deduce that $\sigma(\fq) \ne \lambda_{\wp}\cO_{\wp}$ for every $\sigma \in \Gal(\bK_{\wp}/\k)$. Since $\fq$ is a prime ideal, we know that $\Norm_{\bK_{\wp}/\k}(\fq) = (\fp \A)^r$ for some monic prime $\fp \in \A$ and some positive integer $r$. By \cite[Theorem 3.5.1]{Goldschmidt}, and since  $\wp\cO_{\wp} = (\lambda_{\wp}\cO_{\wp})^{q^{\deg(\wp)} - 1}$, we deduce that $\fp \ne \wp$. By Rosen \cite[Proposition 12.7]{Rosen}, $\bK_{\wp}$ is unramified at $\fp$, and thus $\fp\cO_{\wp}$ is the product of the distinct conjugates of $\fq$. Therefore $\fp\cO_{\wp}$ divides $\cP_{\alpha}(\lambda_{\wp})\cO_{\wp}$, and thus 
\begin{align}
\label{E1-in-L1}
\dfrac{\cP_{\alpha}(\lambda_{\wp})}{\fp} \in \cO_{\wp}.
\end{align}

By Remark \ref{R1-in-L1}, and since $\deg(\cP_{\alpha}(x)) \le q^{\deg(\wp)} - 2$, one can write
\begin{align}
\label{E2-in-L1}
\cP_{\alpha}(x) = \sum_{i = 0}^h \epsilon_i x^i,
\end{align}
where the $\epsilon_i$ are in $\A$ with $\epsilon_h \ne 0$, and $h = \deg(\cP_{\alpha}(x)) \le q^{\deg(\wp)} - 2$. By \cite[Proposition 12.9]{Rosen}, $\cO_{\wp} = \A[\lambda_{\wp}]$, and since the minimal polynomial of $\lambda_{\wp}$ over $\k$ is the $\wp$-th cyclotomic polynomial $\Psi_{\wp}(x) \in \A[x]$ of degree exactly $q^{\deg(\wp)} - 1$ (see Hayes \cite{Hayes} or Section \ref{S-basic-notions}), it follows from (\ref{E1-in-L1}) and (\ref{E2-in-L1}) that there exists an element of the form $\sum_{j = 0}^{q^{\deg(\wp)} - 2}\kappa_j \lambda_{\wp}^j \in \cO_{\wp} = \A[\lambda_{\wp}]$ with the $\kappa_j \in \A$ such that
\begin{align*}
\dfrac{\cP_{\alpha}(\lambda_{\wp})}{\fp} = \sum_{i = 0}^h \dfrac{\epsilon_i}{\fp} \lambda_{\wp}^i = \sum_{j = 0}^{q^{\deg(\wp)} - 2}\kappa_j \lambda_{\wp}^j.
\end{align*}

Therefore
\begin{align}
\label{E3-in-L1}
\sum_{i = 0}^h \left(\dfrac{\epsilon_i}{\fp} - \kappa_i\right) \lambda_{\wp}^i - \sum_{i = h + 1}^{q^{\deg(\wp)} - 2}\kappa_i\lambda_{\wp}^i = 0.
\end{align}
Since $h \le q^{\deg(\wp)} - 2$ and the minimal polynomial of $\lambda_{\wp}$ over $\k$ is the $\wp$-th cyclotomic polynomial $\Psi_{\wp}(x) \in \A[x]$ of degree exactly $q^{\deg(\wp)} - 1$, we deduce from (\ref{E3-in-L1}) that 
\begin{align*}
\dfrac{\epsilon_i}{\fp} - \kappa_i = 0
\end{align*}
for every $0 \le i \le h$. Thus $\epsilon_i = \fp \kappa_i$ for every $0 \le i \le h$. Therefore $\fp$ divides $\gcd_{0 \le i \le h}(\epsilon_i)$. This implies that the content of $\cP_{\alpha}(x)$ is divisible by $\fp$, which is a contradiction to (iii) in Lemma \ref{L1}. Thus $\alpha$ is a unit in $\cO_{\wp}$.

\end{proof}

For each $e \ge 1$, set
\begin{align}
\label{E-rho-e}
\rho_e = \dfrac{\sigma_{\fg}^{e - 1}(\lambda_{\wp})}{\sigma_{\fg}^{e}(\lambda_{\wp})} = \sigma_{\fg}^{e - 1}\left( \dfrac{\lambda_{\wp}}{\sigma_{\fg}(\lambda_{\wp})}\right).
\end{align}
Since $\dfrac{\lambda_{\wp}}{\sigma_{\fg}(\lambda_{\wp})} = \dfrac{\lambda_{\wp}}{C_{\fg}(\lambda_{\wp})}$ is a unit in $\cO_{\wp}$ (see Rosen \cite[Proposition 12.6]{Rosen}), it follows that $\rho_e$ is a unit in $\cO_{\wp}$. Recall from Subsection \ref{Subsection-Galois-group} that $\sigma_{\fg}^r(\lambda_{\wp}) = \sigma_{\fg^r}(\lambda_{\wp}) = C_{\fg^r}(\lambda_{\wp})$; hence one can write (\ref{E-rho-e}) in the form
\begin{align}
\label{E-rho-e-2}
\rho_e = \dfrac{C_{\fg^{e - 1}}(\lambda_{\wp})}{C_{\fg^{e}}(\lambda_{\wp})}.
\end{align}

Let $\cP_{\rho_e}(x) \in \A[x]$ be the polynomial representing $\rho_e$. From the above equation, one gets
\begin{align*}
\cP_{\rho_e}(C_{\fg}(\lambda_{\wp})) = \cP_{\rho_e}(\sigma_{\fg}(\lambda_{\wp})) = \sigma_{\fg}(\cP_{\rho_e}(\lambda_{\wp})) = \sigma_{\fg}(\rho_e) = \sigma_{\fg}\left( \dfrac{\sigma_{\fg}^{e - 1}(\lambda_{\wp})}{\sigma_{\fg}^{e}(\lambda_{\wp})} \right),
\end{align*}
and thus
\begin{align}
\label{Eqn-between-rho-e}
\cP_{\rho_e}(C_{\fg}(\lambda_{\wp})) = \dfrac{\sigma_{\fg}^{e}(\lambda_{\wp})}{\sigma_{\fg}^{e + 1}(\lambda_{\wp})} = \rho_{e + 1} = \cP_{\rho_{e + 1}}(\lambda_{\wp}).
\end{align}

\begin{lemma}
\label{L2}

Let $\ell$ be an integer $\ge 2$, and set
\begin{align*}
\Lambda_{\wp} = \prod_{h= 2}^{\ell}\prod_{e = 2}^h\dfrac{\rho_{e}}{\rho_{e -1}} = \prod_{h= 2}^{\ell}\prod_{e = 2}^h\dfrac{\cP_{\rho_{e - 1}}(C_{\fg}(\lambda_{\wp}))}{\cP_{\rho_{e - 1}}(\lambda_{\wp})}.
\end{align*}
Then
\begin{align}
\label{E-in-L2}
\dfrac{\lambda_{\wp}}{C_{\fg^{\ell}}(\lambda_{\wp})} = \Lambda_{\wp}\left(\dfrac{\lambda_{\wp}}{C_{\fg}(\lambda_{\wp})}\right)^{\ell}.
\end{align}

\end{lemma}

\begin{proof}

For each $h \ge 1$, we see that
\begin{align*}
\prod_{e = 2}^h\dfrac{\rho_{e}}{\rho_{e - 1}} = \dfrac{\rho_h}{\rho_{1}},
\end{align*}
and it thus follows from (\ref{E-rho-e-2}) that
\begin{align*}
\Lambda_{\wp} =  \prod_{h= 2}^{\ell}\dfrac{\rho_h}{\rho_{1}} = \dfrac{1}{\rho_1^{\ell - 1}} \prod_{h= 2}^{\ell}\dfrac{C_{\fg^{h - 1}}(\lambda_{\wp})}{C_{\fg^{h}}(\lambda_{\wp})} = \left(\dfrac{C_{\fg}(\lambda_{\wp})}{\lambda_{\wp}}\right)^{\ell - 1}\dfrac{C_{\fg}(\lambda_{\wp})}{C_{\fg^{\ell}}(\lambda_{\wp})}.
\end{align*}
Therefore (\ref{E-in-L2}) follows immediately.

\end{proof}

Let $\ell = q^{\deg(\wp)} - 1$. Note that $(\star)$ in Subsection \ref{Subsection-Galois-group} implies that $\ell \ge 2$. Since 
\begin{align*}
C_{\fg^{ q^{\deg(\wp)} - 1}}(\lambda_{\wp}) = \sigma_{\fg^{ q^{\deg(\wp)} - 1}}(\lambda_{\wp}) = \sigma_{\fg}^{q^{\deg(\wp)} - 1}(\lambda_{\wp}) = \lambda_{\wp},
\end{align*}
we deduce from Lemma \ref{L2} that
\begin{align}
\label{E-for-Lambda_wp-in-terms-of-lambda_wp}
\Lambda_{\wp} = \left(\dfrac{\lambda_{\wp}}{C_{\fg}(\lambda_{\wp})}\right)^{1 - q^{\deg(\wp)}}.
\end{align}

We summarize the above discussion in the following result.

\begin{corollary}
\label{C1}

Let $\ell$ be an integer. Then there exist an integer $e$ with $0 \le e \le q^{\deg(\wp)} - 2$ and an integer $h$ such that the following are true:
\begin{itemize}

\item [(i)] $\ell = (q^{\deg(\wp)} - 1)h + e$; and

\item [(ii)]
\begin{align*}
\left(\dfrac{\lambda_{\wp}}{C_{\fg}(\lambda_{\wp})}\right)^{\ell} = \left(\dfrac{\lambda_{\wp}}{C_{\fg}(\lambda_{\wp})}\right)^{e}\Lambda_{\wp}^{-h}.
\end{align*}

\end{itemize}

\end{corollary}

From the definition of $\Lambda_{\wp}$ in Lemma \ref{L2}, $\Lambda_{\wp}^{-1}$ clearly belongs to the multiplicative group $\cG_{\wp}$, and so does $\Lambda_{\wp}^h$ for any integer $h$. (Recall that $\cG_{\wp}$ is defined in Subsection \ref{Subsection-G-wp}.) From Proposition \ref{P4} and Corollary \ref{C1}, we obtain the following result.

\begin{corollary}
\label{C2}

Let $\ell$ be an integer. Then there exist an integer $e$ with $0 \le e \le q^{\deg(\wp)} - 2$ and an element $\dfrac{\cQ(\lambda_{\wp})}{\cQ(C_{\fg}(\lambda_{\wp}))} \in \cG_{\wp}$ such that
\begin{align*}
\left(\dfrac{\lambda_{\wp}}{C_{\fg}(\lambda_{\wp})}\right)^{\ell} = \left(\dfrac{\lambda_{\wp}}{C_{\fg}(\lambda_{\wp})}\right)^{e}\dfrac{\cQ(\lambda_{\wp})}{\cQ(C_{\fg}(\lambda_{\wp}))}.
\end{align*}

\end{corollary}

\begin{lemma}
\label{L3}

Let $\cQ(x)$ be a polynomial in $\A[x]$ such that $\cQ(\lambda_{\wp})$ is a unit in $\cO_{\wp}$. Then $\gcd(\cQ(0), \wp) = 1$.

\end{lemma}

\begin{proof}

Let $\epsilon = \cQ(\lambda_{\wp})$. By assumption, $\epsilon$ is a unit in $\cO_{\wp}$. Assume the contrary, i.e., $\wp$ divides $\cQ(0)$.  It is known (see Rosen \cite[Proposition 12.7]{Rosen}) that $\wp\cO_{\wp} = (\lambda_{\wp}\cO_{\wp})^{q^{\deg(\wp)} - 1}$. Thus $\lambda_{\wp}$ divides $\cQ(0)$, and hence
\begin{align*}
\epsilon = \cQ(\lambda_{\wp}) \equiv \cQ(0) \equiv 0 \pmod{\lambda_{\wp}},
\end{align*}
which is a contradiction since $\epsilon$ is a unit in $\cO_{\wp}$ and $\lambda_{\wp}\cO_{\wp}$ is a prime ideal in $\cO_{\wp}$ (see Rosen \cite[Proposition 12.7]{Rosen}).

\end{proof}

The next result is our main theorem in this paper, which can be viewed as a function field analogue of Newman \cite[Theorem]{Newman}. The proof of the next theorem follows the same ideas as in the proof of Newman \cite[Theorem]{Newman}, but we need some modifications to adapt the proof of Newman \cite[Theorem]{Newman} into the function field setting.

\begin{theorem}
\label{MT}

Let $\epsilon$ be a unit in $\cO_{\wp}$ of norm $1$. Then there exist an integer $\ell$ with $0 \le \ell \le q^{\deg(\wp)} - 2$, and a polynomial $\cQ(x) \in \A[x]$ of degree at most $q^{\deg(\wp)} - 2$ with $\cQ(\lambda_{\wp})$ being a unit in $\cO_{\wp}$ such that the following are true:
\begin{itemize}

\item [(R1)] $\epsilon$ can be represented in the form
\begin{align*}
\epsilon = \left(\dfrac{\lambda_{\wp}}{C_{\fg}(\lambda_{\wp})}\right)^{\ell}\left(\dfrac{\cQ(\lambda_{\wp})}{\cQ(C_{\fg}(\lambda_{\wp}))}\right).
\end{align*}

\item [(R2)] The representation of $\epsilon$ in (R1) is unique, except that $\cQ(x)$ can be replaced by $\upsilon\cQ(x)$ for some unit $\upsilon \in \bF_q^{\times}$; more precisely, if there exist an integer $\ell_{\star}$ with $0 \le \ell_{\star} \le q^{\deg(\wp)} - 2$, and a  polynomial $\cQ_{\star}(x) \in \A[x]$ of degree at most $q^{\deg(\wp)} - 2$ with $\cQ_{\star}(\lambda_{\wp})$ being a unit in $\cO_{\wp}$ such that 
\begin{align*}
\epsilon = \left(\dfrac{\lambda_{\wp}}{C_{\fg}(\lambda_{\wp})}\right)^{\ell_{\star}}\left(\dfrac{\cQ_{\star}(\lambda_{\wp})}{\cQ_{\star}(C_{\fg}(\lambda_{\wp}))}\right),
\end{align*}
then $\ell_{\star} = \ell$ and $\cQ_{\star}(x) = \upsilon \cQ(x)$ for some unit $\upsilon \in \bF_q^{\times}$.

\end{itemize}

\end{theorem}

\begin{proof}

Let $\cP_{\epsilon}(x) \in \A[x]$ be the polynomial representing $\epsilon$, i.e., $\epsilon = \cP_{\epsilon}(\lambda_{\wp})$. By Lemma \ref{L3}, 
\begin{align}
\label{E1-MT}
\gcd(\cP_{\epsilon}(0), \wp) = 1.
\end{align}

By (\ref{E1-MT}) and since $\fg$ is a generator of the cyclic group $(\A/\wp\A)^{\times}$, and $\#(\A/\wp\A)^{\times} = q^{\deg(\wp)} - 1$, there exists an integer $h$ with $0 \le h \le q^{\deg(\wp)} - 2$ such that 
\begin{align}
\label{E2-MT}
\cP_{\epsilon}(0) \equiv \fg^h \pmod{\wp}.
\end{align}

Set
\begin{align}
\label{E3-MT}
\epsilon = \left(\dfrac{C_{\fg}(\lambda_{\wp})}{\lambda_{\wp}}\right)^{h - 1}\alpha
\end{align}
for some $\alpha \in \bK_{\wp}$. From Rosen \cite[Proposition 12.6]{Rosen}, we know that $\dfrac{C_{\fg}(\lambda_{\wp})}{\lambda_{\wp}}$ is a unit in $\cO_{\wp}$, and thus $\alpha$ is a unit in $\cO_{\wp}$. 

Let $\cP_{\alpha}(x) \in \A[x]$ be the polynomial representing $\alpha$, and set
\begin{align}
\label{E4-MT}
\cQ(x) = \left(\dfrac{C_{\fg}(x)}{x}\right)^{q^{\deg(\wp)} + h - 2}\cP_{\alpha}(x).
\end{align}
From Rosen \cite[Proposition 12.11]{Rosen}, one can write $C_{\fg}(x) \in \A[x]$ in the form
\begin{align*}
C_{\fg}(x) = \fg x + [\fg, 1]x^q + \ldots + [\fg, \deg(\fg) - 1]x^{q^{\deg(\fg) - 1}} + [\fg, \deg(\fg)]x^{q^{\deg(\fg)}},
\end{align*}
where $[\fg, i] \in \A$ for each $1 \le i \le \deg(\fg)$. Thus
\begin{align}
\label{E5-MT}
\dfrac{C_{\fg}(x)}{x} = \fg  + [\fg, 1]x^{q - 1} + \ldots + [\fg, \deg(\fg)]x^{q^{\deg(\fg)} - 1} \in \A[x].
\end{align}
Since $q^{\deg(\wp)} + h - 2 > 0$, we see that $\cQ(x) \in \A[x]$.

We see from (\ref{E3-MT}) that 
\begin{align*}
\cQ(\lambda_{\wp}) &= \left(\dfrac{C_{\fg}(\lambda_{\wp})}{\lambda_{\wp}}\right)^{q^{\deg(\wp)} + h - 2}\cP_{\alpha}(\lambda_{\wp}) \\
&= \left(\dfrac{C_{\fg}(\lambda_{\wp})}{\lambda_{\wp}}\right)^{q^{\deg(\wp)} -1}\epsilon,
\end{align*}
and thus
\begin{align*}
\cQ(\lambda_{\wp}) - \left(\dfrac{C_{\fg}(\lambda_{\wp})}{\lambda_{\wp}}\right)^{q^{\deg(\wp)} -1}\cP_{\epsilon}(\lambda_{\wp}) = \cQ(\lambda_{\wp}) - \left(\dfrac{C_{\fg}(\lambda_{\wp})}{\lambda_{\wp}}\right)^{q^{\deg(\wp)} -1}\epsilon =0.
\end{align*}
Since $\cQ(x) - \left(\dfrac{C_{\fg}(x)}{x}\right)^{q^{\deg(\wp)} -1}\cP_{\epsilon}(x) \in \A[x]$, and the $\wp$-th cyclotomic polynomial $\Psi_{\wp}(x)$ is the minimal polynomial of $\lambda_{\wp}$, it follows from the above equation that $\Psi_{\wp}(x)$ divides $\cQ(x) - \left(\dfrac{C_{\fg}(x)}{x}\right)^{q^{\deg(\wp)} -1}\cP_{\epsilon}(x)$ in $\A[x]$. Thus there exists a polynomial, say $\cT(x) \in \A[x]$ such that
\begin{align}
\label{E6-MT}
\cQ(x) - \left(\dfrac{C_{\fg}(x)}{x}\right)^{q^{\deg(\wp)} -1}\cP_{\epsilon}(x) = \cT(x)\Psi_{\wp}(x).
\end{align}

Setting $\cZ(x) = \dfrac{C_{\fg}(x)}{x} \in \A[x]$, we see from (\ref{E5-MT}) that
\begin{align}
\label{E7-MT}
\cZ(0) = \fg.
\end{align}

From Proposition \ref{P-Psi(0)=wp}, and (\ref{E6-MT}), we deduce that
\begin{align*}
\cQ(0) - \cZ(0)^{q^{\deg(\wp)} -1}\cP_{\epsilon}(0) = \cT(0)\Psi_{\wp}(0) = \cT(0)\wp \equiv 0 \pmod{\wp}.
\end{align*}
From (\ref{E2-MT}), \eqref{E4-MT}, \eqref{E7-MT}, and the above equation, we deduce that
\begin{align*}
 \fg^{q^{\deg(\wp)} + h - 2}\cP_{\alpha}(0) = \cZ(0)^{q^{\deg(\wp)} + h - 2}\cP_{\alpha}(0) = \cQ(0) \equiv \cZ(0)^{q^{\deg(\wp)} -1}\cP_{\epsilon}(0) \equiv \fg^{q^{\deg(\wp)} + h -1} \pmod{\wp},
\end{align*}
and thus
\begin{align}
\label{E8-MT}
\cP_{\alpha}(0) \equiv \fg \pmod{\wp}.
\end{align}

We know that $\sigma_{\fg}(\lambda_{\wp}) = C_{\fg}(\lambda_{\wp})$, and thus $\dfrac{C_{\fg}(\lambda_{\wp})}{\lambda_{\wp}} = \dfrac{\sigma_{\fg}(\lambda_{\wp})}{\lambda_{\wp}}$. Since $\sigma_{\fg}$ is a generator of the Galois group $\Gal(\bK_{\wp}/\k)$, it follows that 
\begin{align*}
\Nm_{\bK_{\wp}/\k}\left(\dfrac{C_{\fg}(\lambda_{\wp})}{\lambda_{\wp}} \right) = \Nm_{\bK_{\wp}/\k}\left( \dfrac{\sigma_{\fg}(\lambda_{\wp})}{\lambda_{\wp}}\right) = 1.
\end{align*}
Since $\epsilon$ is of norm $1$, equation \eqref{E3-MT} and the above equation imply that $\alpha$ is also of norm $1$. Applying Proposition \ref{P-Hilbert-Satz-90-for-units}, there exists an element $\delta \in \cO_{\wp}$ such that
\begin{align}
\label{E9-MT}
\alpha = \dfrac{\delta}{\sigma_{\fg}(\delta)}.
\end{align}

Let $\cP_{\delta}(x) \in \A[x]$ be the polynomial representing $\delta$. Since $\delta = \cP_{\delta}(\lambda_{\wp})$, one can write
\begin{align}
\label{E9-1/2-MT}
\cP_{\delta}(\lambda_{\wp}) = \delta = \lambda_{\wp}^e\eta,
\end{align}
where $\eta \in \cO_{\wp}$ such that $\gcd(\lambda_{\wp}, \eta) = 1$, and $e$ is a nonnegative integer. 

Let $\cP_{\eta}(x) \in \A[x]$ be the polynomial representing $\eta$. Since 
\begin{align*}
\eta = \cP_{\eta}(\lambda_{\wp}) \equiv \cP_{\eta}(0) \pmod{\lambda_{\wp}}, 
\end{align*}
and $\wp\cO_{\wp} = (\lambda_{\wp}\cO_{\wp})^{q^{\deg(\wp)}- 1}$, we see that if $\wp$ divides $\cP_{\eta}(0)$, then the prime ideal $\lambda_{\wp}\cO_{\wp}$ divides $\eta$, which is a contradiction. Hence
\begin{align}
\label{E10-MT}
\gcd(\cP_{\eta}(0), \wp) = 1.
\end{align}

Note that the polynomial $\cP_{\eta}(x)$ is not identical to zero since $\eta \ne 0$. Let $\fc(\cP_{\eta}) \in \A$ be the content of the polynomial $\cP_{\eta}(x)$, i.e., the greatest common divisor of all nonzero coefficients of $\cP_{\eta}(x)$ in $\A$. Obviously $\fc(\cP_{\eta}) \ne 0$. Set 
\begin{align}
\label{E11-MT}
\cR(x) = \dfrac{\cP_{\eta}(x)}{\fc(\cP_{\eta})} \in \A[x],
\end{align}
and let 
\begin{align}
\label{E12-MT}
\beta = \cR(\lambda_{\wp}) =  \dfrac{\cP_{\eta}(\lambda_{\wp})}{\fc(\cP_{\eta})} = \dfrac{\eta}{\fc(\cP_{\eta})} \in \cO_{\wp} = \A[\lambda_{\wp}].
\end{align}
Let $\cP_{\beta}(x) \in \A[x]$ be the polynomial representing $\beta$. Since $\deg(\cR(x)) = \deg(\cP_{\eta}(x)) \le q^{\deg(\wp)} - 2$, the uniqueness implies that $\cP_{\beta}(x) = \cR(x)$ (see Proposition \ref{P2}). It follows from (\ref{E11-MT}) that
\begin{align}
\label{E13-MT}
\cP_{\beta}(x) = \dfrac{\cP_{\eta}(x)}{\fc(\cP_{\eta})}.
\end{align}

From (\ref{E9-MT}) and (\ref{E9-1/2-MT}), and since $\fc(\cP_{\eta}) \in \A$, one sees that
\begin{align*}
\alpha = \dfrac{\lambda_{\wp}^e\cP_{\eta}(\lambda_{\wp})}{\sigma_{\fg}(\lambda_{\wp}^e\cP_{\eta}(\lambda_{\wp}))} = \dfrac{\lambda_{\wp}^e\cP_{\eta}(\lambda_{\wp})}{\sigma_{\fg}(\lambda_{\wp}^e)\sigma_{\fg}(\cP_{\eta}(\lambda_{\wp}))} = \left(\dfrac{\lambda_{\wp}}{\sigma_{\fg}(\lambda_{\wp})}\right)^e\dfrac{\fc(\cP_{\eta})\cP_{\beta}(\lambda_{\wp})}{\sigma_{\fg}(\fc(\cP_{\eta})\cP_{\beta}(\lambda_{\wp}))} = \left(\dfrac{\lambda_{\wp}}{C_{\fg}(\lambda_{\wp})}\right)^e\dfrac{\cP_{\beta}(\lambda_{\wp})}{\sigma_{\fg}(\cP_{\beta}(\lambda_{\wp}))},
\end{align*}
and thus
\begin{align*}
\alpha  =  \left(\dfrac{\lambda_{\wp}}{C_{\fg}(\lambda_{\wp})}\right)^e \dfrac{\cP_{\beta}(\lambda_{\wp})}{\cP_{\beta}(\sigma_{\fg}(\lambda_{\wp}))} =  \left(\dfrac{\lambda_{\wp}}{C_{\fg}(\lambda_{\wp})}\right)^e\dfrac{\cP_{\beta}(\lambda_{\wp})}{\cP_{\beta}(C_{\fg}(\lambda_{\wp}))}.
\end{align*}
It therefore follows from (\ref{E3-MT}) that
\begin{align}
\label{E14-MT}
\epsilon =  \left(\dfrac{\lambda_{\wp}}{C_{\fg}(\lambda_{\wp})}\right)^{e - h + 1}\dfrac{\cP_{\beta}(\lambda_{\wp})}{\cP_{\beta}(C_{\fg}(\lambda_{\wp}))}. 
\end{align}

By Rosen \cite[Proposition 12.6]{Rosen}, one knows that $\dfrac{\lambda_{\wp}}{C_{\fg}(\lambda_{\wp})}$ is a unit in $\cO_{\wp}$, and thus 
\begin{itemize}

\item[(i)] $\dfrac{\cP_{\beta}(\lambda_{\wp})}{\cP_{\beta}(C_{\fg}(\lambda_{\wp}))}$ is a unit in $\cO_{\wp}$, or equivalently $\dfrac{\cP_{\beta}(C_{\fg}(\lambda_{\wp}))}{\cP_{\beta}(\lambda_{\wp})}$ is a unit in $\cO_{\wp}$.

\end{itemize}

From (\ref{E10-MT}), (\ref{E13-MT}), and since $\fc(\cP_{\eta})$ is the content of $\cP_{\eta}(x)$, we deduce that

\begin{itemize}

\item [(ii)] $\gcd(\cP_{\beta}(0), \wp) = 1$.

\item [(iii)] the content of the polynomial $\cP_{\beta}(x)$ is 1.

\end{itemize}

Since $\cP_{\beta}(x)$ satisfies all the conditions in Lemma \ref{L1}, we deduce that $\beta = \cP_{\beta}(\lambda_{\wp})$ is a unit in $\cO_{\wp}$, and thus $\dfrac{\cP_{\beta}(\lambda_{\wp})}{\cP_{\beta}(C_{\fg}(\lambda_{\wp}))}$ belongs to the group $\cG_{\wp}$. (Recall that $\cG_{\wp}$ is defined in Subsection \ref{Subsection-G-wp}.) By Corollary \ref{C2}, there exist an integer $\ell$ with $0 \le \ell \le q^{\deg(\wp)} - 2$, and an element $\dfrac{\cP(\lambda_{\wp})}{\cP(C_{\fg}(\lambda_{\wp}))} \in \cG_{\wp}$ such that
\begin{align*}
\left(\dfrac{\lambda_{\wp}}{C_{\fg}(\lambda_{\wp})}\right)^{e - h + 1} = \left(\dfrac{\lambda_{\wp}}{C_{\fg}(\lambda_{\wp})}\right)^{\ell}\dfrac{\cP(\lambda_{\wp})}{\cP(C_{\fg}(\lambda_{\wp}))},
\end{align*}
and it thus follows from (\ref{E14-MT}) that
\begin{align}
\label{E16-MT}
\epsilon = \left(\dfrac{\lambda_{\wp}}{C_{\fg}(\lambda_{\wp})}\right)^{\ell} \left(\dfrac{\cP(\lambda_{\wp})}{\cP(C_{\fg}(\lambda_{\wp}))}\dfrac{\cP_{\beta}(\lambda_{\wp})}{\cP_{\beta}(C_{\fg}(\lambda_{\wp}))}\right).
\end{align}

Since $\cG_{\wp}$ is a multiplicative group, it follows from Proposition \ref{P3} and Proposition \ref{P4}(iii) that there exists a polynomial $\cQ(x) \in \A[x]$ of degree at most $q^{\deg(\wp)} - 2$ such that  $\cQ(\lambda_{\wp})$ is a unit in $\cO_{\wp}$, and 
\begin{align*}
\dfrac{\cP(\lambda_{\wp})}{\cP(C_{\fg}(\lambda_{\wp}))}\dfrac{\cP_{\beta}(\lambda_{\wp})}{\cP_{\beta}(C_{\fg}(\lambda_{\wp}))} = \dfrac{\cQ(\lambda_{\wp})}{\cQ(C_{\fg}(\lambda_{\wp}))}.
\end{align*}
Hence we deduce from (\ref{E16-MT}) that
\begin{align*}
\epsilon = \left(\dfrac{\lambda_{\wp}}{C_{\fg}(\lambda_{\wp})}\right)^{\ell}\left(\dfrac{\cQ(\lambda_{\wp})}{\cQ(C_{\fg}(\lambda_{\wp}))}\right),
\end{align*}
which proves the first part of the theorem.

We now prove the uniqueness. Assume that there exist another integer $\ell_{\star}$ with $0 \le \ell_{\star} \le q^{\deg(\wp)} - 2$, and another polynomial $\cQ_{\star}(x) \in \A[x]$ of degree at most $q^{\deg(\wp)} - 2$ with $\cQ_{\star}(\lambda_{\wp})$ being a unit in $\cO_{\wp}$ such that
\begin{align}
\label{E17-MT}
\epsilon = \left(\dfrac{\lambda_{\wp}}{C_{\fg}(\lambda_{\wp})}\right)^{\ell}\left(\dfrac{\cQ(\lambda_{\wp})}{\cQ(C_{\fg}(\lambda_{\wp}))}\right) = \left(\dfrac{\lambda_{\wp}}{C_{\fg}(\lambda_{\wp})}\right)^{\ell_{\star}}\left(\dfrac{\cQ_{\star}(\lambda_{\wp})}{\cQ_{\star}(C_{\fg}(\lambda_{\wp}))}\right).
\end{align}
Hence
\begin{align}
\label{E18-MT}
\left(\dfrac{C_{\fg}(\lambda_{\wp})}{\lambda_{\wp}}\right)^{\ell}\left(\dfrac{\cQ(C_{\fg}(\lambda_{\wp}))}{\cQ(\lambda_{\wp})}\right) = \left(\dfrac{C_{\fg}(\lambda_{\wp})}{\lambda_{\wp}}\right)^{\ell_{\star}}\left(\dfrac{\cQ_{\star}(C_{\fg}(\lambda_{\wp}))}{\cQ_{\star}(\lambda_{\wp})}\right).
\end{align}

By Rosen \cite[Proposition 12.11]{Rosen}, $C_{\fg}(x) \in \A[x]$ can be written in the form
\begin{align*}
C_{\fg}(x) = \fg x + [\fg, 1]x^q + \ldots + [\fg, \deg(\fg)] x^{q^{\deg(\fg)}},
\end{align*}
where the $[\fg, i]$ are in $\A$, and $[\fg, \deg(\fg)] \in \bF_q^{\times}$ is the leading coefficient of $\fg$. It follows that 
\begin{align}
\label{E19-MT}
C_{\fg}(0) = 0,
\end{align}
and 
\begin{align}
\label{E20-MT}
\dfrac{C_{\fg}(\lambda_{\wp})}{\lambda_{\wp}} \equiv \fg \pmod{\lambda_{\wp}}.
\end{align}
Furthermore we deduce from Lemma \ref{L3} that $\gcd(\cQ(0), \wp) = 1$ and $\gcd(\cQ_{\star}(0), \wp) = 1$. In particular this implies that $\cQ(0), \cQ_{\star}(0)$ are nonzero elements in $\A$. Hence it follows from (\ref{E18-MT}), (\ref{E19-MT}), and (\ref{E20-MT}) that
\begin{align*}
\fg^{\ell} \equiv \fg^{\ell_{\star}} \pmod{\lambda_{\wp}}.
\end{align*}

Since $\wp\cO_{\wp} = (\lambda_{\wp}\cO_{\wp})^{q^{\deg(\wp)} - 1}$ (see Rosen \cite[Proposition 12.7]{Rosen}), we deduce that
\begin{align*}
\fg^{\ell} \equiv \fg^{\ell_{\star}} \pmod{\wp}.
\end{align*}
Without loss of generality, we assume that $\ell \ge \ell_{\star}$. If $\ell > \ell_{\star}$, then
\begin{align*}
\fg^{\ell - \ell_{\star}} \equiv 1 \pmod{\wp},
\end{align*}
and since $\fg$ is a primitive root modulo $\wp$, the above equation implies that $\# (\A/\wp \A)^{\times} = q^{\deg(\wp)} - 1$ divides $\ell - \ell_{\star}$. This is a contradiction since $0 < \ell - \ell_{\star} < q^{\deg(\wp)} - 1$. Hence 
\begin{align}
\label{E21-MT}
\ell = \ell_{\star}.
\end{align}

By (\ref{E17-MT}), and since $\sigma_{\fg}(\lambda_{\wp}) = C_{\fg}(\lambda_{\wp})$ , we deduce that
\begin{align*}
\dfrac{\cQ(\lambda_{\wp})}{\cQ_{\star}(\lambda_{\wp})} = \dfrac{ \cQ(C_{\fg}(\lambda_{\wp}))}{\cQ_{\star}(C_{\fg}(\lambda_{\wp}))} = \sigma_{\fg}\left(\dfrac{ \cQ(\lambda_{\wp})}{\cQ_{\star}(\lambda_{\wp})}\right).
\end{align*}
Thus the unit $\dfrac{\cQ(\lambda_{\wp})}{\cQ_{\star}(\lambda_{\wp})}$ in $\cO_{\wp}$ is invariant under the action of $\sigma_{\fg}$. Since $\sigma_{\fg}$ is a generator of the cyclic group $\Gal(\bK_{\wp}/\k)$, the unit $\dfrac{\cQ(\lambda_{\wp})}{\cQ_{\star}(\lambda_{\wp})}$ is also invariant under the action of each member of $\Gal(\bK_{\wp}/\k)$, and thus $\dfrac{\cQ(\lambda_{\wp})}{\cQ_{\star}(\lambda_{\wp})}$ is a unit in $\A$. This implies that $\dfrac{\cQ(\lambda_{\wp})}{\cQ_{\star}(\lambda_{\wp})} \in \bF_q^{\times}$, and therefore there exists a unit $\upsilon \in \bF_q^{\times}$ such that
\begin{align}
\label{E22-MT}
\cQ_{\star}(\lambda_{\wp}) = \upsilon \cQ(\lambda_{\wp}).
\end{align}

Set $\kappa = \cQ_{\star}(\lambda_{\wp})$. By the assumption, we know that $\kappa$ is a unit in $\cO_{\wp}$. Since the polynomial $\cQ_{\star}(x)$ is of degree at most $q^{\deg(\wp)} - 2$, it follows from Proposition \ref{P2} that $\cQ_{\star}(x) = \cP_{\kappa}(x)$, where $\cP_{\kappa}(x) \in \A[x]$ is the polynomial representing $\kappa$. 

On the other hand, we know from (\ref{E22-MT}) that $\kappa = \upsilon \cQ(\lambda_{\wp})$. Note that $\upsilon \cQ(x)$ is a polynomial in $\A[x]$ of degree at most  $q^{\deg(\wp)} - 2$ since $\upsilon \in \bF_q^{\times}$. Hence repeating the same arguments as above, one sees that $\upsilon\cQ(x) = \cP_{\kappa}(x)$, and thus
\begin{align}
\label{E23-MT}
\cQ_{\star}(x) = \upsilon \cQ(x).
\end{align}

The second part of the theorem follows immediately from (\ref{E21-MT}) and (\ref{E23-MT}).

\end{proof}

Recall that the multiplicative group $\cG_{\wp}$ consists of all elements $\dfrac{\cQ(\lambda_{\wp})}{\cQ(C_{\fg}(\lambda_{\wp}))}$, where $\cQ(x)$ is a polynomial in $\A[x]$ such that $\cQ(\lambda_{\wp})$ is a unit in $\cO_{\wp}$. The next result follows immediately from Theorem \ref{MT}.

\begin{corollary}
\label{C1-MT}

Let $\epsilon$ be a unit in $\cO_{\wp}$ of norm $1$. Then there exist an integer $\ell$ with $0 \le \ell \le q^{\deg(\wp)} - 2$, and an element $\kappa \in \cG_{\wp}$ such that the following are true:

\begin{itemize}

\item [(i)] $\epsilon$ can be represented in the form
\begin{align*}
\epsilon = \left(\dfrac{\lambda_{\wp}}{C_{\fg}(\lambda_{\wp})}\right)^{\ell}\kappa.
\end{align*}

\item [(ii)] The representation of $\epsilon$ in (i) is unique; more precisely, if there exist another integer $\ell_{\star}$ with $0 \le \ell_{\star} \le q^{\deg(\wp)} - 2$, and another element $\kappa_{\star} \in \cG_{\wp}$ such that
\begin{align*}
\epsilon = \left(\dfrac{\lambda_{\wp}}{C_{\fg}(\lambda_{\wp})}\right)^{\ell_{\star}}\kappa_{\star},
\end{align*}
then $\ell_{\star} = \ell$ and $\kappa_{\star} = \kappa$.

\end{itemize}

\end{corollary}

Corollary \ref{C1-MT} can be interpreted in terms of group theory.

\begin{corollary}
\label{C2-MT}

\begin{itemize}

\item []

\item [(i)] $\cG_{\wp}$ is a subgroup of $\cU_{\star}(\cO_{\wp})$ such that $[\cU_{\star}(\cO_{\wp}) : \cG_{\wp}] = q^{\deg(\wp)} - 1$. (Recall that $\cU_{\star}(\cO_{\wp})$ is the group of units of norm $1$ in $\cO_{\wp}$.)

\item [(ii)] The quotient group $\cU_{\star}(\cO_{\wp}) / \cG_{\wp}$ is cyclic of order $q^{\deg(\wp)} - 1$; furthermore $\cU_{\star}(\cO_{\wp}) / \cG_{\wp}$ is generated by $\dfrac{\lambda_{\wp}}{C_{\fg}(\lambda_{\wp})}\cG_{\wp}$.

\end{itemize}

\end{corollary}

Hilbert's Satz 90 (see Lang \cite{Lang}) tells us that for a unit $\epsilon \in \cU_{\star}(\cO_{\wp})$, there exists an element $\delta \in \cO_{\wp}$ such that $\epsilon = \dfrac{\delta}{\sigma_{\fg}(\delta)}$ (see Proposition \ref{P-Hilbert-Satz-90-for-units}). The next result is a refinement of Hilbert's Satz 90 for the extension $\bK_{\wp}/\k$, which gives a sufficient and necessary condition under which an element in $\cO_{\wp}$ is a quotient of conjugate units. 

\begin{corollary}
\label{C3-MT}

Let $\epsilon \in \cU_{\star}(\cO_{\wp})$. Then $\epsilon$ is the quotient of conjugate units in $\cO_{\wp}$ if and only if $\cP_{\epsilon}(0) \equiv 1 \pmod{\wp}$, where $\cP_{\epsilon}(x) \in \A[x]$ is the polynomial representing $\epsilon$.

\end{corollary}

\begin{proof}

Assume that $\epsilon$ is the quotient of conjugate units in $\cO_{\wp}$, i.e., $\epsilon = \dfrac{\delta}{\sigma_{\fg^e}(\delta)}$ for some unit $\delta \in \cO_{\wp}$, and some integer $e$ with $0 \le e \le q^{\deg(\wp)} - 2$. (Recall that $\sigma_{\fg}$ is a generator of the Galois group $\Gal(\bK_{\wp}/\k)$.) If $e = 0$, then $\epsilon = 1$, and it is easy to see that the polynomial $\cP_{\epsilon}(x) = \cP_1(x) = 1$. Hence $\cP_1(0) \equiv 1 \pmod{\wp}$. 

We now consider the case where $e \ge 1$. Set
\begin{align*}
\alpha = \prod_{i = 0}^{e - 1}\sigma_{\fg^i}(\delta) = \delta \cdot \sigma_{\fg}(\delta)\cdots\sigma_{\fg^{e - 1}}(\delta).
\end{align*}
One can immediately verify that $\alpha$ is a unit in $\cO_{\wp}$, and
\begin{align}
\label{E1-C3}
\epsilon = \dfrac{\delta}{\sigma_{\fg^e}(\delta)} = \dfrac{\alpha}{\sigma_{\fg}(\alpha)}.
\end{align}

By Proposition \ref{P4}, $\dfrac{\alpha}{\sigma_{\fg}(\alpha)}$ belongs to the group $\cG_{\wp}$. Since $\epsilon = \cP_{\epsilon}(\lambda_{\wp})$, we deduce from (\ref{E1-C3}) that
\begin{align}
\label{E2-C3}
 \cP_{\epsilon}(\lambda_{\wp}) = \dfrac{\alpha}{\sigma_{\fg}(\alpha)} = \dfrac{\cP_{\alpha}(\lambda_{\wp})}{\cP_{\alpha}(C_{\fg}(\lambda_{\wp}))},
\end{align}
where $\cP_{\alpha}(x)$ is the polynomial representing $\alpha$. Following the same arguments as in the proof of Theorem \ref{MT} (see equation (\ref{E19-MT})), we deduce that $C_{\fg}(0) = 0$. Using Lemma \ref{L3}, we know that $\gcd(\cP_{\alpha}(0), \wp) = 1$;  hence $\gcd(\cP_{\alpha}(0), \lambda_{\wp}) = 1$ since $\wp\cO_{\wp} = (\lambda_{\wp}\cO_{\wp})^{q^{\deg(\wp)} - 1}$ and $\lambda_{\wp}\cO_{\wp}$ is a prime ideal in $\cO_{\wp}$ (see Rosen \cite[Proposition 12.7]{Rosen}). In particular, this implies that $\cP_{\alpha}(0) \not\equiv 0 \pmod{\lambda_{\wp}}$. By (\ref{E2-C3}), we see that
\begin{align*}
\cP_{\epsilon}(0) \equiv \cP_{\epsilon}(\lambda_{\wp}) = \dfrac{\cP_{\alpha}(\lambda_{\wp})}{\cP_{\alpha}(C_{\fg}(\lambda_{\wp}))} \equiv  \dfrac{\cP_{\alpha}(0)}{\cP_{\alpha}(C_{\fg}(0))} = \dfrac{\cP_{\alpha}(0)}{\cP_{\alpha}(0)} = 1 \pmod{\lambda_{\wp}},
\end{align*}
and thus 
\begin{align*}
\cP_{\epsilon}(0) \equiv 1 \pmod{\wp}.
\end{align*}

Conversely assume that $\epsilon \in \cU_{\star}(\cO_{\wp})$ such that 
\begin{align}
\label{E2-1/2-C3}
\cP_{\epsilon}(0) \equiv 1 \pmod{\wp}. 
\end{align}

By Theorem \ref{MT}, one can write $\epsilon$ in the form
\begin{align}
\label{E3-C3}
\epsilon = \cP_{\epsilon}(\lambda_{\wp}) = \left(\dfrac{\lambda_{\wp}}{C_{\fg}(\lambda_{\wp})}\right)^{\ell}\dfrac{\cQ(\lambda_{\wp})}{\cQ(C_{\fg}(\lambda_{\wp}))},
\end{align}
where $\ell$ is an integer such that $0 \le \ell \le q^{\deg(\wp)} - 2$, and $\cQ(x)$ is a polynomial in $\A[x]$ of degree at most $q^{\deg(\wp)} - 2$ such that $\cQ(\lambda_{\wp})$ is a unit in $\cO_{\wp}$. One can write (\ref{E3-C3}) in the form
\begin{align}
\label{E4-C3}
\left(\dfrac{ C_{\fg}(\lambda_{\wp})}{\lambda_{\wp}}\right)^{\ell}\cP_{\epsilon}(\lambda_{\wp}) = \dfrac{\cQ(\lambda_{\wp})}{\cQ(C_{\fg}(\lambda_{\wp}))}.
\end{align}

Following the same arguments as in the proof of Theorem \ref{MT} (see equation (\ref{E20-MT})), we know that
\begin{align*}
\dfrac{C_{\fg}(\lambda_{\wp})}{\lambda_{\wp}} \equiv \fg \pmod{\lambda_{\wp}},
\end{align*}
and thus
\begin{align}
\label{E5-C3}
\left(\dfrac{ C_{\fg}(\lambda_{\wp})}{\lambda_{\wp}}\right)^{\ell}\equiv \fg^{\ell} \pmod{\lambda_{\wp}}.
\end{align}
Repeating the same arguments as in the first part of this proof, one can show that 
\begin{align*}
\dfrac{\cQ(\lambda_{\wp})}{\cQ(C_{\fg}(\lambda_{\wp}))} \equiv \dfrac{\cQ(0)}{\cQ(C_{\fg}(0))} = \dfrac{\cQ(0)}{\cQ(0)} = 1 \pmod{\lambda_{\wp}},
\end{align*}
and it thus follows from (\ref{E4-C3}) and (\ref{E5-C3}) that
\begin{align}
\label{E6-C3}
\fg^{\ell}\cP_{\epsilon}(0) \equiv 1 \pmod{\lambda_{\wp}}.
\end{align}

Since $\fg^{\ell}\cP_{\epsilon}(0) \in \A$, and $\wp\cO_{\wp} = (\lambda_{\wp}\cO_{\wp})^{q^{\deg(\wp)} - 1}$ (see Rosen \cite[Proposition 12.7]{Rosen}), we deduce from (\ref{E6-C3}) that
\begin{align*}
\fg^{\ell}\cP_{\epsilon}(0) \equiv 1 \pmod{\wp},
\end{align*}
and it thus follows from (\ref{E2-1/2-C3}) that
\begin{align*}
\fg^{\ell} \equiv 1 \pmod{\wp}.
\end{align*}
Since $\fg$ is a generator of the cyclic group $(\A/\wp\A)^{\times}$, and the order of $(\A/\wp\A)^{\times}$ is $q^{\deg(\wp)} - 1$, the above equation implies that $\ell = 0$. Therefore we deduce from (\ref{E3-C3}) that
\begin{align*}
\epsilon = \dfrac{\cQ(\lambda_{\wp})}{\cQ(C_{\fg}(\lambda_{\wp}))} =  \dfrac{\cQ(\lambda_{\wp})}{\cQ(\sigma_{\fg}(\lambda_{\wp}))} =  \dfrac{\cQ(\lambda_{\wp})}{\sigma_{\fg}(\cQ(\lambda_{\wp}))} =  \dfrac{\kappa}{\sigma_{\fg}(\kappa)},
\end{align*}
where $\kappa = \cQ(\lambda_{\wp})$ is a unit in $\cO_{\wp}$. Thus our contention follows.

\end{proof}

\begin{remark}

Corollary \ref{C3-MT} is a function field analogue of Newman \cite[Corollary, p.357]{Newman}.

\end{remark}

\end{document}